\documentclass[oneside,english]{amsart}
\usepackage[T1]{fontenc}
\usepackage[latin9]{inputenc}
\usepackage{geometry}
\usepackage{amsthm}
\usepackage{amssymb}
\usepackage{setspace}
\makeatletter
\numberwithin{equation}{section}
\numberwithin{figure}{section}
\theoremstyle{plain}
\newtheorem{thm}{\protect\theoremname}
  \theoremstyle{definition}
  \newtheorem{defn}[thm]{\protect\definitionname}
  \theoremstyle{remark}
  \newtheorem{rem}[thm]{\protect\remarkname}
  \theoremstyle{plain}
  \newtheorem{lem}[thm]{\protect\lemmaname}
  \theoremstyle{plain}
  \newtheorem{cor}[thm]{\protect\corollaryname}
  \theoremstyle{plain}
  \newtheorem{prop}[thm]{\protect\propositionname}

\makeatother

\usepackage{babel}
  \providecommand{\corollaryname}{Corollary}
  \providecommand{\definitionname}{Definition}
  \providecommand{\lemmaname}{Lemma}
  \providecommand{\propositionname}{Proposition}
  \providecommand{\remarkname}{Remark}
\providecommand{\theoremname}{Theorem}

\begin{document}

\title[Girsanov Theorem for Multifractional Brownian Processes]{Girsanov Theorem for Multifractional Brownian Processes }

\author{Fabian A. Harang, Torstein Nilssen, Frank N. Proske}

\address{Fabian A. Harang\\
Email: fabianah@math.uio.no\\
Department of Mathematics\\
University of Oslo\\
Postboks 1053\\
Blindern 0316\\
Oslo\\
 }

\address{Torstein Nilssen\\
Email: torsteka@math.uio.no\\
Department of Mathematics\\
University of Oslo\\
Postboks 1053\\
Blindern 0316\\
Oslo\\
Funded by Norwegian Research Council (Project 230448/F20)}

\address{Frank N. Proske\\
Email: proske@math.uio.no\\
Department of Mathematics\\
University of Oslo\\
Postboks 1053\\
Blindern 0316\\
Oslo}
\begin{abstract}
In this article we will present a new perspective on the variable
order fractional calculus, which allows for differentiation and integration
to a variable order, i.e. one differentiates (or integrates) a function
along the path of a ``regularity function''. The concept of multifractional
calculus has been a scarcely studied topic within the field of functional
analysis in the last 20 years. We develop a multifractional derivative
operator which acts as the inverse of the multifractional integral
operator. This is done by solving the Abel integral equation generalized
to a multifractional order. With this new multifractional derivative
operator, we are able to analyze a variety of new problems, both in
the field of stochastic analysis and in fractional and functional
analysis, ranging from regularization properties of noise to solutions
to multifractional differential equations. In this paper, we will
focus on application of the derivative operator to the construction
of strong solutions to stochastic differential equations where the
drift coefficient is merely of linear growth, and the driving noise
is given by a non-stationary multifractional Brownian motion with
a Hurst parameter as a function of time. The Hurst functions we study
will take values in a bounded subset of $(0,\frac{1}{2})$. The application
of multifractional calculus to SDE's is based on a generalization
of the works of D. Nualart and Y. Ouknine in \cite{NuaOki} from 2002.
\end{abstract}


\maketitle
\tableofcontents{}

\section{Introduction }

Fractional calculus is a well studied subject in the field of functional
analysis, and has been adopted in the field of stochastic analysis
as a tool for analysing stochastic processes which exhibit some memory
of its own past, called fractional Brownian motion (fBm), denoted
as $\tilde{B}_{\cdot}^{H}:[0,T]\rightarrow\mathbb{R}$ for some $H\in(0,1)$.
The process is characterized by its co-variance function 
\[
R_{H}(s,t)=E[\tilde{B}_{s}^{H}\tilde{B}_{t}^{H}]=\frac{1}{2}\left(t^{2H}+s^{2H}-|t-s|^{2H}\right),
\]
and the process is often represented as a type of fractional integral
with respect to a Brownian motion, in the sense that 
\[
\tilde{B}_{t}^{H}=\int_{0}^{t}K_{H}(t,s)dB_{s},
\]
where $K_{H}$ is a singular and square integrable Volterra kernel
and $B_{\cdot}$ is a Brownian motion. This type of process has been
studied in connection with various physical phenomena, such as weather
and geology, but also in modelling of internet traffic and finance.

In the 1990's, a generalization of the fractional Brownian motion
was proposed by letting the Hurst parameter $H$ of the process $B_{\cdot}^{H}$
to be dependent on time (e.g. \cite{Peltier}). The process was named
multifractional Brownian motion (or mBm for short), referring to the
fact that the fractional parameter $H$ was a function depending on
time taking values between $0$ and $1$. There has later been a series
of articles on this type of processes, relating to local time, local
and global H\"{o}lder continuity of the process, and other applications
( see \cite{LebVehHer,LebovitsLevy,Corlay,BouDozMar,Bianchi} to name
a few). However, there has not been, to the best of our knowledge,
any articles on solving stochastic differential equations driven by
this type of process. The mBm is a non-stationary stochastic process
which makes the it more intricate to handle in differential equations,
and require sufficient tools from fractional analysis in the sense
of multifractional calculus (or fractional calculus of variable order).
This type of fractional calculus was originally proposed by S. Samko
(\cite{Samko1} ) and others in the beginning of the 1990's, and generalizes
the classical Riemann Liouville fractional integral by considering
\[
I_{0+}^{\alpha}f(x)=\frac{1}{\Gamma(\alpha(x))}\int_{0}^{x}(x-t)^{\alpha(x)-1}f(t)dt,
\]
where $f\in L^{1}\left(\mathbb{R}\right)$ and $\alpha:\mathbb{R}\rightarrow[c,d]\subset(0,1)$.
In the same way, the authors also proposed to generalize the fractional
derivative i.e. 
\[
D_{0+}^{\alpha}f(x)=\frac{1}{\Gamma\left(1-\alpha(x)\right)}\frac{d}{dx}\int_{0}^{x}\frac{f(t)}{(x-t)^{\alpha(x)}}dt.
\]
 However, by generalizing the fractional derivative in this way, the
authors found that the derivative is no longer the inverse of the
integral operator, but one rather finds that 
\[
D_{0+}^{\alpha}I_{0+}^{\alpha}=I+K,
\]
where $I$ is the identity, and under some conditions $K$ is a compact
operator. Some of the research then focused on understanding the properties
of the operator $K$, but fractional calculus of variable order has
been scarcely treated in the literature, so far. 

It is well known in stochastic analysis that one can construct weak
solutions to SDE's by Girsanov's theorem, see for example \cite{NuaOki}
in the case when the SDE is driven by a fBm. To be able to develop
a Girsanov's theorem to SDE's driven by multifractional noise, we
need to have an inverse operator relating to the multifractional integral. 

We will in this paper construct the multifractional derivative of
a function $f\in L^{p}\left([0,T]\right)$ as the inverse of the multifractional
integral, and see that this multifractional derivative operator is
well defined on a certain class of functions. We will then apply this
operator to the construction of a Girsanov theorem for Riemann-Liouville
multifractional Brownian motions, and finally we show existence of
strong solutions to certain SDE's. 

The SDE we will investigate is given by 
\begin{equation}
X_{t}^{x}=x+\int_{0}^{t}b\left(s,X_{s}^{x}\right)ds+B_{t}^{h};\,\,X_{0}^{x}=x\in\mathbb{R},\label{eq:The SDE}
\end{equation}
where $B^{h}$ is a Riemann Liouville multifractional Brownian motion
with $regularity$ $function$ $h:[0,T]\rightarrow[a,b]\subset(0,\frac{1}{2})$.
Actually, it turns out that we only need $t\mapsto\int_{0}^{t}b\left(s,X_{s}^{x}\right)ds$
to be locally $\beta$ - H\" {o}lder continuous for some well chosen $\beta$
to construct weak solutions, but we will need $b$ to be of linear
growth to get strong solutions by the comparison theorem. 

\subsection{Notation and preliminaries}

We will make use of a space of H\" {o}lder continuous functions defined
as functions $f:[0,T]\rightarrow V$ for some Banach space $V$ (we
will mostly use $V=\mathbb{R}^{d}$ in this article), which is such
that the H\" {o}lder norm $\parallel f\parallel_{\beta}$ defined by 
\[
\parallel f\parallel_{\beta}=\left|f(0)\right|+\sup_{s\neq t\in[0,T]}\frac{\left|f(t)-f(s)\right|}{\left|t-s\right|^{\beta}}<\infty.
\]
We denote this space by $\mathcal{C}^{\beta}\left([0,T];V\right)$.
In addition, for $a\in[0,T]$ let us define $\mathcal{C}_{a}^{\beta}\left([0,T];V\right)$
to be the subspace of $\mathcal{C}^{\beta}\left([0,T];V\right)$ such
that for $f\in\mathcal{C}_{a}^{\beta}\left([0,T];V\right)$, $f(a)=0$.
In particular, we will be interested in $\mathcal{C}_{0}^{\beta}\left([0,T];V\right)$.
We use the standard notation for $L^{p}\left(V;\mu\right)$ spaces,
where $V$ is a Banach space and $\mu$ is a measure on the Borel
sets of $V$ (usually what measure we use is clear from the situation)i.e.,
\[
\parallel f\parallel_{L^{p}\left(V\right)}:=\left(\int_{V}|f|^{p}d\mu\right)^{\frac{1}{p}}.
\]
 Furthermore let $\triangle^{(m)}\left[a,b\right]$ denote the $m$-simplex.
That is, define $\triangle^{(m)}[a,b]$ to be given by 
\[
\triangle^{(m)}[a,b]=\left\{ (s_{1},..,s_{m})|a\leq s_{1}<...<s_{m}\leq b\right\} .
\]
\\
We will use the notion of variable order exponent spaces, which has
in recent years become increasingly popular in potential analysis,
see for example \cite{SamRafMesKok,DieHarHasRuz} for two new books
on the subject. We will in this paper be particularly concerned with
the variable exponent H\" {o}lder space, as we are looking to differentiate
functions to a variable order. We will therefore follow the definitions
of such a space introduced in \cite{SamRos}, and give some preliminary
properties. 
\begin{defn}
Let $\alpha$ be a $C^{1}$ regularity function with values in $[a,b]\subset(0,1)$,
as described above. We define the space of locally H\" {o}lder continuous
functions $f:[0,T]\rightarrow\mathbb{R}$ by the norm 
\[
\parallel f\parallel_{\alpha\left(\cdot\right);[0,T]}:=|f(0)|+\sup_{x,y\in[0,T]}\frac{|f(x)-f(y)|}{|x-y|^{\max\left(\alpha(x),\alpha(y)\right)}}<\infty
\]
We denote this space by $\mathcal{C}^{\alpha\left(\cdot\right)}\left([0,T];\mathbb{R}\right)$.
Moreover denote by $\mathcal{C}_{0}^{\alpha(\cdot)}\left([0,T];\mathbb{R}\right)$
the space of locally H\" {o}lder continuous functions which start in $0,$
i.e. for $f\in\mathcal{C}_{0}^{\alpha(\cdot)}\left([0,T];\mathbb{R}\right)$
then $f(0)=0$.
\end{defn}
\begin{rem}
Notice that in $\mathcal{C}_{0}^{\alpha(\cdot)}\left([0,T];\mathbb{R}\right)$
, all functions $f$ satisfy $|f(x)|\leq Cx^{\alpha(x)}.$ Indeed,
just write 
\[
|f(x)|\leq|f(x)-f(0)|+f(0)\leq\parallel f\parallel_{\alpha\left(\cdot\right);[0,T]}x^{\alpha(x)},
\]
where we have used that $\max(\alpha\left(x\right),\alpha\left(y\right))\geq\alpha\left(x\right).$
Further, we have the following equivalence 
\[
\sup_{|h|\leq1;x+h\in[0,T]}\frac{|f(x+h)-f(x)|}{|h|^{\alpha(x)}}\simeq\sup_{x,y\in[0,T]}\frac{|f(x)-f(y)|}{|x-y|^{\max\left(\alpha(x),\alpha(y)\right)}}.
\]
\end{rem}
To show what we mean by the local regular property, we may divide
the interval $[0,T]$ into $n\geq2$ intervals, by defining $T_{0}=0$,
$T_{1}=\frac{T}{n}$ and $T_{k}=\frac{kT}{n}$, then 
\[
[0,T]=\bigcup_{k=0}^{n-1}[T_{k},T_{k+1}].
\]
Furthermore, we may define two sequences $\{\beta_{k}\}$ and $\{\hat{\beta}_{k}\}$
of $\sup$ and $\inf$ values of $\alpha$ restricted to each interval
$[T_{k},T_{k+1}]$, i.e 
\[
\begin{array}{cc}
\beta_{k}= & \inf_{t\in[T_{k},T_{k+1}]}\alpha(t)\\
\hat{\beta}_{k}= & \sup_{t\in[T_{k},T_{k+1}]}\alpha(t)
\end{array}
\]
 Then we have the following inclusions 
\[
\mathcal{C}^{\beta_{k}}\left(\left[T_{k},T_{k+1}\right]\right)\supset\mathcal{C}^{\alpha(\cdot)}\left(\left[T_{k},T_{k+1}\right]\right)\supset\mathcal{C}^{\hat{\beta}_{k}}\left(\left[T_{k},T_{k+1}\right]\right).
\]
By letting $n$ be a large number, we have by the continuity of $\alpha$
that $\hat{\beta}_{k}-\beta_{k}=\varepsilon$ gets small. Therefore
the space $\mathcal{C}^{\alpha(\cdot)}\left(\left[T_{k},T_{k+1}\right]\right)$
is similar to the spaces $\mathcal{C}^{\beta_{k}}\left(\left[T_{k},T_{k+1}\right]\right)$
and $\mathcal{C}^{\hat{\beta}_{k}}\left(\left[T_{k},T_{k+1}\right]\right)$
when the $T_{k+1}-T_{k}$ is small, but on large intervals the spaces
may be different depending on the chosen $\alpha$. 

We will often write $f_{*}:=\inf_{t\in[0,T]}f(t)$ and $f^{*}:=\sup_{t\in[0,T]}f(t)$. 

\section{multifractional Calculus \label{sec:Multifractional-Calculus}}

In this section we will give meaning to the multifractional calculus.
Many references in fractional analysis refer to this concept as fractional
calculus of variable order, as the idea is to let the order of integration
(or differentiation) be dependent on time (or possibly space, but
for our application, time will be sufficient), see for example \cite{Samko1,Samko2}
and the references therein. We will use the word multifractional calculus
for the concept of fractional calculus of variable order, as this
is coherent with the notion of multifractional stochastic processes.
\\

\begin{defn}
\label{Multifractional Rieman Liouville Integrals) }($Multifractional$
$Riemann$-$Liouville$ $integrals$) For $0<c<d$, assume $f\in L^{1}\left([c,d]\right),$
and $\alpha:[c,d]\rightarrow[a,b]\subset(0,\infty)$ is a differentiable
function. We define the left multifractional Riemann Liouville integral
operator $I_{c+}^{\alpha}$ by 
\[
\left(I_{c+}^{\alpha}f\right)(x)=\frac{1}{\Gamma\left(\alpha(x)\right)}\int_{c}^{x}(x-y)^{\alpha(x)-1}f(y)dy.
\]
 And define the space $I_{c+}^{\alpha}L^{p}\left([0,T]\right)$ as
the image of $L^{p}\left([0,T]\right)$ under the operator $I_{c+}^{\alpha}.$ 
\end{defn}
By the definition of the space $I_{c+}^{\alpha}L^{p}\left([0,T]\right),$
we have that for all $g\in I_{c+}^{\alpha}L^{p}\left([0,T]\right)$,
$g(c)=0.$ Indeed, since $g=I_{c+}^{\alpha}f,$ we must have $I_{c+}^{\alpha}f(c)=0$
regardless of the function $\alpha$. This property will become very
important later, and we will give a more thorough discussion of the
properties of the multifractional integral and derivative at the end
of this section. \\

We want to define the multifractional derivative of a function $g\in I_{c+}^{\alpha}L^{p}\left([0,T]\right)$
as the inverse operation of $I_{c+}^{\alpha}$, such that if $g\in I_{c+}^{\alpha}L^{p}\left([0,T]\right)$
then there exists a unique $f\in L^{p}$ which satisfies $g=I_{c+}^{\alpha}f$
and we define the fractional derivative $D_{c+}^{\alpha}g=f$. In
contrast to this methodology used in \cite{Samko1}, we believe that
to be able to construct a coherent fractional calculus, one must choose
to generalize either the derivative or the integral operator, and
the other operator must be found through the definition of the first.
Therefore, we will use a method similar to that solving Abel's integral
equation, often used to motivate the definition of the fractional
derivative in the case of constant regularity function. However, by
generalizing the integral equation, the calculations become a little
more complicated.\\
 \\
We will start to give formal motivation for how we obtain the derivative
operator corresponding to the multifractional integral. First, for
an element $g\in I_{a+}^{\alpha}L^{p}\left([0,T]\right)$, we know
there exists a $f\in L^{p}\left([0,T]\right)$ such that 
\[
g(t)=\frac{1}{\Gamma\left(\alpha(t)\right)}\int_{a}^{t}\frac{f(s)}{\left(t-s\right)^{1-\alpha(t)}}ds.
\]
By some simple manipulation of the equation above, and using Fubini's
theorem, we obtain the equation, 
\[
\int_{a}^{x}\frac{\Gamma\left(\alpha(t)\right)g(t)}{\left(x-t\right)^{\alpha(t)}}dt=\int_{a}^{x}f(s)\int_{0}^{1}\tau^{\alpha(s+\tau(x-s))-1}\left(1-\tau\right)^{-\alpha(s+\tau(x-s))}d\tau ds.
\]
If $\alpha$ is constant, the last integral is simply the Beta function.
However, the fact that $\alpha$ is a function complicates the expression.
In the end, we are interested in solving the above equation by obtaining
$f(x)$, therefore, we may try to differentiate w.r.t $x$ on both
sides of the equation. We find that 
\[
\frac{d}{dx}\left(\int_{a}^{x}\frac{\Gamma\left(\alpha(t)\right)g(t)}{\left(x-t\right)^{\alpha(t)}}dt\right)
\]
\begin{equation}
=f(x)B(\alpha(x),1-\alpha(x))+\int_{a}^{x}f(s)\int_{0}^{1}\frac{d}{dx}\left(\tau^{\alpha(s+\tau(x-s))-1}\left(1-\tau\right)^{-\alpha(s+\tau(x-s))}\right)d\tau ds,\label{The ODE}
\end{equation}
and re-ordering the terms, we get 
\[
B(\alpha(x),1-\alpha(x))f(x)=\frac{d}{dx}\left(\int_{a}^{x}\frac{\Gamma\left(\alpha(t)\right)g(t)}{\left(x-t\right)^{\alpha(t)}}dt\right)
\]
\[
-\int_{a}^{x}f(s)\int_{0}^{1}\frac{d}{dx}\left(\tau^{\alpha(s+\tau(x-s))-1}\left(1-\tau\right)^{-\alpha(s+\tau(x-s))}\right)d\tau ds.
\]
 In this sense, the multifractional derivative $D_{a+}^{\alpha}$
of a function $g\in I_{a+}^{\alpha}L^{p}\left([0,T]\right)$ is given
by $f$ which is a solution to the ordinary differential equation
(ODE) above. Writing the above equation more compactly, we have 
\begin{equation}
f(x)=G_{a}(g)(x)+\int_{a}^{x}f(s)F(s,x)ds.\label{eq:Multifractional ODE}
\end{equation}

We will use the rest of this section to prove that the ODE that we
obtained above, actually is well defined, and that the multifractional
derivative can be defined as the solution to that ODE. In the above
motivation we looked at the multifractional derivative as the inverse
operator to the multifractional integral with initial value of integration
at a point $a$. However, we will mostly be interested in looking
at the inverse operator of the integral starting at $a=0$, and therefore
the construction of the multifractional will be focused on this particular
case, although it is straight forward to generalize this operator
to any $a\in[0,T]$. We will give a suitable definition of the multifractional
derivative in the following steps: 
\begin{enumerate}
\item The function $F:\triangle^{(2)}\left([0,T]\right)\rightarrow\mathbb{R}$
in equation (\ref{eq:Multifractional ODE}) is well defined and bounded
on $\triangle^{(2)}\left([0,T]\right)$ as long as $\alpha$ is $C^{1}$. 
\item We show that if $H\in L^{p}\left([0,T]\right),$ then the ODE 
\[
f(x)=H(x)+\int_{0}^{x}f(s)F(s,x)ds,
\]
has a unique solution $f$ in $L^{p}\left([0,T]\right).$
\item The functional $G_{0}$ defined by 
\[
G_{0}\left(g\right)(x)=\frac{1}{B(\alpha(x),1-\alpha(x))}\frac{d}{dx}\left(\int_{0}^{x}\frac{\Gamma\left(\alpha(t)\right)g(t)}{\left(x-t\right)^{\alpha(t)}}dt\right),
\]
 is such that the mapping $x\mapsto G_{0}\left(g\right)(x)\in L^{p}\left([0,T]\right)$
for all functions $g\in\mathcal{C}^{\alpha\left(\cdot\right)+\epsilon}\left([0,T]\right)$,
where $p>1$ and some small $\epsilon>0$ depending on $\alpha$. 
\item Then the ODE as motivated above has a unique solution in $L^{p}\left([0,T]\right)$
for all $g\in\mathcal{C}^{\alpha\left(\cdot\right)+\epsilon}\left([0,T]\right)$,
and we will define the multifractional derivative $D_{0+}^{\alpha}g=f$
as the solution to the equation (\ref{eq:Multifractional ODE}). \\
\\
\end{enumerate}
First, we will look at the function $F:\triangle^{(2)}\left([0,T]\right)\rightarrow\mathbb{R}$
in equation (\ref{eq:Multifractional ODE}). Notice that the derivative
of $\tau^{\alpha(s+\tau(x-s))-1}\left(1-\tau\right)^{-\alpha(s+\tau(x-s))}$
is explicitly given by 
\[
\frac{d}{dx}\left(\tau^{\alpha(s+\tau(x-s))-1}\left(1-\tau\right)^{-\alpha(s+\tau(x-s))}\right)
\]
\[
=\ln(\frac{\tau}{1-\tau})\frac{\tau^{\alpha(s+\tau(x-s))}}{\left(1-\tau\right)^{\alpha(s+\tau(x-s))}}\alpha'(s+\tau(x-s)),
\]
 and we arrive at a lemma which gives estimates on this derivative. 
\begin{lem}
\label{Let existence of beta derivative integral }Let $\alpha\in C^{1}\left(\left[0,T\right];(0,1)\right)$.
Define 
\[
U(s,x;\tau):=\alpha'(s+\tau(x-s))\times\left(\ln(\tau)-\ln\left(1-\tau\right)\right)\left(\frac{\tau}{1-\tau}\right)^{\alpha(s+\tau(x-s))},\,\,\,and
\]
\[
F(s,x)=\int_{0}^{1}U(s,x;\tau)d\tau.
\]
Then 
\[
\parallel F\parallel_{\infty;\triangle^{(2)}\left([0,T]\right)}:=\sup_{(s,x)\in\triangle^{(2)}([0,T])}\left|\int_{0}^{1}U(s,x;\tau)d\tau\right|\leq C(\alpha)\parallel\alpha\parallel_{C^{1}\left([0,T]\right)}<\infty,
\]
where $\parallel\alpha\parallel_{C^{1}\left([0,T]\right)}:=\sup_{t\in[0,T]}|\alpha(t)|+\sup_{t\in[0,T]}|\alpha'(t)|.$
\end{lem}
\begin{proof}
Look at, 
\[
\left|\int_{0}^{1}\left(\ln(\tau)-\ln\left(1-\tau\right)\right)\left(\frac{\tau}{1-\tau}\right)^{\alpha(s+\tau(x-s))}\alpha'(s+\tau(x-s))d\tau\right|
\]
\[
\leq\parallel\alpha\parallel_{C^{1}\left([0,T]\right)}\int_{0}^{1}\left|\left(\ln(\tau)-\ln\left(1-\tau\right)\right)\left(\frac{\tau}{1-\tau}\right)^{\alpha(s+\tau(x-s))}\right|d\tau<\infty.
\]
The fact that the last integral is finite follows by simple calculations,
knowing that $\alpha(t)\in[c,d]\subset(0,1)$ for all $t\in[0,T].$
Indeed, we can write 
\[
\int_{0}^{1}|\ln(\tau)|\left(\frac{\tau}{1-\tau}\right)^{\alpha(s+\tau(x-s))}d\tau+\int_{0}^{1}|\ln(1-\tau)|\left(\frac{\tau}{1-\tau}\right)^{\alpha(s+\tau(x-s))}d\tau
\]
\[
=\int_{0}^{\frac{1}{2}}|\ln(\tau)|\left(\frac{\tau}{1-\tau}\right)^{\alpha(s+\tau(x-s))}d\tau+\int_{0}^{\frac{1}{2}}|\ln(1-\tau)|\left(\frac{\tau}{1-\tau}\right)^{\alpha(s+\tau(x-s))}d\tau
\]
\[
+\int_{\frac{1}{2}}^{1}|\ln(\tau)|\left(\frac{\tau}{1-\tau}\right)^{\alpha(s+\tau(x-s))}d\tau+\int_{\frac{1}{2}}^{1}|\ln(1-\tau)|\left(\frac{\tau}{1-\tau}\right)^{\alpha(s+\tau(x-s))}d\tau
\]
\[
=I_{1}+..+I_{4}
\]
 Now we have
\[
I_{1}\leq2^{\alpha^{*}}\int_{0}^{\frac{1}{2}}|\ln(\tau)|\tau^{\alpha_{*}}d\tau<\infty
\]
\[
I_{2}\leq2^{\alpha^{*}}|\ln\left(\frac{1}{2}\right)|\int_{0}^{\frac{1}{2}}\tau^{\alpha_{*}}d\tau<\infty
\]
\[
I_{3}\leq2^{\alpha^{*}}|\ln\left(\frac{1}{2}\right)|\int_{0}^{\frac{1}{2}}\left(1-\tau\right)^{-\alpha^{*}}d\tau<\infty
\]
\[
I_{4}\leq2^{1-\alpha_{*}}|\ln\left(\frac{1}{2}\right)|\int_{0}^{\frac{1}{2}}\left(1-\tau\right)^{-\alpha^{*}}d\tau<\infty,
\]
where $\alpha_{*}:=\inf_{t\in[0,T]}\alpha(t)$ and $\alpha^{*}:=\sup_{t\in[0,T]}\alpha(t)$. 
\end{proof}
We now show that the solution to the differential equation given by
\[
f(x)=H(x)+\int_{0}^{x}f(s)F(s,x)ds
\]
 is well-defined as an element of $L^{p}\left(\left[0,T\right]\right)$
when $H\in L^{p}\left([0,T]\right)$, and $F$ is given as above. 
\begin{lem}
Let $H\in L^{p}\left(\left[0,T\right]\right)$ and $F$ be given as
in Lemma \ref{Let existence of beta derivative integral }. Then there
exists a unique solution $f$ to the equation 
\[
f(x)=H(x)+\int_{0}^{x}f(s)F(s,x)ds,
\]
in $L^{p}\left(\left[0,T\right]\right)$. 
\end{lem}
\begin{proof}
For simplicity, write $L^{p}=L^{p}\left([0,T]\right)$. We consider
a usual Picard iteration, and define 
\[
\begin{array}{cc}
f_{0}(x)=H(x)\\
f_{n}(x)=H(x)+\int_{0}^{x}f_{n-1}(s)F(s,x)ds,
\end{array}
\]
 then for $x\in[0,T]$ we use Lemma \ref{Let existence of beta derivative integral }
to see that
\[
\left\Vert f_{n+1}-f_{n}\right\Vert _{L^{p}}^{p}=\left\Vert \int_{0}^{\cdot}\left(f_{n}(s)-f_{n-1}(s)\right)F(s,\cdot)ds\right\Vert _{L^{p}}^{p}
\]
\[
\leq\parallel F\parallel_{\infty;\triangle^{(2)}\left([0,T]\right)}^{p}\int_{0}^{T}x^{p-1}\left\Vert f_{n}-f_{n-1}\right\Vert _{L^{p}\left(\left[0,x\right]\right)}^{p}dx,
\]
where we have used the H\" {o}lder inequality . By iteration, we obtain
\[
\left\Vert f_{n+1}-f_{n}\right\Vert _{L^{p}}\leq\parallel F\parallel_{\infty;\triangle^{(2)}\left([0,T]\right)}^{np}\parallel H\parallel_{L^{p}}\int_{0}^{T}...\int_{0}^{x_{n}}x_{n}^{p-1}dx_{n}...dx_{1}
\]
\[
=\parallel F\parallel_{\infty;\triangle^{(2)}\left([0,T]\right)}^{np}\frac{\Gamma\left(p\right)}{\Gamma\left(p+n\right)}T^{n+p-1}\parallel H\parallel_{L^{p}}
\]
and more generally that for $m>n$ 
\[
\left\Vert f_{m}-f_{n}\right\Vert _{L^{p}}=\left\Vert \sum_{i=n}^{m-1}f_{i+1}-f_{i}\right\Vert _{L^{p}}
\]
\[
\leq\parallel H\parallel_{L^{p}}\sum_{i=n}^{m-1}\parallel F\parallel_{\infty;\triangle^{(2)}\left([0,T]\right)}^{ip}\frac{\Gamma\left(p\right)}{\Gamma\left(p+i\right)}T^{i+p-1}.
\]
Therefore, $\left\{ f_{n}\right\} _{n\in\mathbb{N}}$ is Cauchy in
$L^{p}$ and$\left\Vert f_{m}-f_{n}\right\Vert _{L^{p}}\rightarrow0$
as $n,m\rightarrow\infty$ , and we denote $f^{\infty}$ to be its
limit. Furthermore, we can choose a sub sequence $f_{n_{k}}$ converging
almost surely and in $L^{p}$ to $f^{\infty}$. By applying dominated
convergence theorem, $f^{\infty}$ solves the original ODE, 
\[
f^{\infty}(x)=H(x)+\int_{0}^{x}f^{\infty}(s)F(s,x)ds\,\,\,in\,\,\,L^{p}\left([0,T]\right),
\]
and we set $f:=f^{\infty}.$ The uniqueness of the solution is immediate
by the linearity of the equation, and can bee seen explicit from the
representation in Theorem \ref{thm.:representation and boundedness of derivative-1}.
\end{proof}
The next lemma shows that the functional $G_{0}$ in equation (\ref{eq:Multifractional ODE})
is an element of $L^{p}$ when acting on a certain class of functions. 
\begin{lem}
\label{lem:eksistens av G0 operator}Let $g\in\mathcal{C}_{0}^{\alpha(\cdot)+\epsilon}\left([0,T];\mathbb{R}\right)$
with $\alpha\in C^{1}\left([0,T],[a,b]\right)$ for $[a,b]\subset(0,1)$
and $\epsilon<1-\alpha^{*}$. Furthermore, assume that for some $p>1,$
the regularity function $\alpha$ satisfies the inequality 
\[
\left(\alpha_{*}+\epsilon-\alpha(0)\right)\times p>-1.
\]
 Then the functional $G_{0}$ evaluated in $g$ defined by 
\[
G_{0}\left(g\right)(x)=\frac{1}{B(\alpha(x),1-\alpha(x))}\frac{d}{dx}\left(\int_{0}^{x}\frac{\Gamma\left(\alpha(t)\right)g(t)}{\left(x-t\right)^{\alpha(t)}}dt\right)
\]
is an element of $L^{p}\left(\left[0,T\right]\right)$.
\end{lem}
\begin{proof}
Let us first look at $G_{0}(g)(x)$, but ignoring the factor with
the Beta function as this is behaving well already, i.e, there exists
a $\beta>0$ such that 
\[
0<\beta\leq B(\alpha(x),1-\alpha(x))\,\,\,for\,\,\,all\,\,\,x\in[0,T]
\]
 as $\alpha(x)\in[a,b]$ for all $x\in[0,T]$. Therefore, we need
to prove that 
\[
\tilde{G}_{0}(g)(x)=\frac{d}{dx}\left(\int_{0}^{x}\frac{\Gamma\left(\alpha(t)\right)g(t)}{\left(x-t\right)^{\alpha(t)}}dt\right)=:\frac{d}{dx}\left(\int_{0}^{x}K(x,t)g(t)dt\right)
\]
 is an element of $L^{p}\left([0,T]\right).$ In particular, we will
show that the integral above is in fact continuously differentiable
on $(0,T],$ and then we show that it is an element in $L^{p}\left([0,T]\right)$.
Expanding the integral by adding and subtracting the point $g(x)$,
we can see that, 
\[
\int_{0}^{x}K(x,t)g(t)dt=g(x)\int_{0}^{x}K(x,t)dt-\int_{0}^{x}K(x,t)\left(g(x)-g(t)\right)dt.
\]
Then one can show that 
\[
\frac{d}{dx}\left(\int_{0}^{x}K(x,t)g(t)dt\right)
\]
\[
=g(x)\frac{d}{dx}\int_{0}^{x}K(x,t)dt-\int_{0}^{x}\frac{d}{dx}K(x,t)\left(g(x)-g(t)\right)dt.
\]
 See the appendix in section \ref{sec:Apendix} for proof a of this
relation, which is based on using the definition of the derivative.
In this way, we can make sense of this derivative by proving that
the above two terms are elements of $L^{p}\left(\left[0,T\right]\right)$
when $g\in\mathcal{C}^{\alpha\left(\cdot\right)+\epsilon}\left(\left[0,T\right]\right)$.
\\
\\
Writing the second term explicitly, we have by straight forward derivation
that 
\[
-\int_{0}^{x}\frac{d}{dx}K(x,t)\left(g(x)-g(t)\right)dt=\int_{0}^{x}\alpha(t)\frac{\Gamma\left(\alpha(t)\right)\left(g(x)-g(t)\right)}{(x-t)^{\alpha(t)+1}}dt.
\]
 Notice that $\frac{d}{dx}K(x,t)$ is  singular, but since $g\in\mathcal{C}_{0}^{\alpha(\cdot)+\epsilon}\left([0,T]\right)$,
we get the estimate
\[
\left|\int_{0}^{x}\frac{d}{dx}K(x,t)\left(g(x)-g(t)\right)dt\right|
\]
\[
\leq\parallel g\parallel_{\mathcal{C}^{\alpha(\cdot)+\epsilon}\left([0,T]\right)}\times\int_{0}^{x}\alpha(t)\left|\Gamma\left(\alpha(t)\right)\right|\left|x-t\right|^{\max(\alpha(x),\alpha(t))-\alpha(t)+\epsilon-1}dt=:\parallel g\parallel_{\mathcal{C}^{\alpha(\cdot)+\epsilon}\left([0,T]\right)}P_{1}(x,\epsilon).
\]
The gamma function is well defined on any $[a,b]\subset(0,1)$, and
$\alpha$ is $C^{1}$ and we know 
\[
\max\left(\alpha\left(x\right),\alpha\left(t\right)\right)-\alpha\left(t\right)=\max\left(\alpha(x)-\alpha(t),0\right)\geq0,
\]
 we have that the singularity is integrable, hence the above is well
defined for all $x\in[0,T]$. The function $P_{1}$ is dependent on
both $x$ and $\epsilon$, but also on $\alpha$, and we can see that
\[
\left|P_{1}(x,\epsilon)\right|\leq C(\alpha)\frac{x^{\epsilon}}{\epsilon}.
\]
 Therefore, $P_{1}$ is $L^{p}$ integrable, for all $p$ as long
as $\epsilon>0$. \\
\\
We are left to prove that the first term is an element of $L^{p}$,
i.e., we need to show that 
\[
x\mapsto g(x)\frac{d}{dx}\int_{0}^{x}\frac{\Gamma\left(\alpha(t)\right)}{(x-t)^{\alpha(t)}}dt\in L^{p}\left(\left[0,T\right]\right).
\]
 Write $M(x)=\int_{0}^{x}\frac{\Gamma\left(\alpha(t)\right)}{(x-t)^{\alpha(t)}}dt,$
and set $u=x-t,$ then by a variable change $M(x)=\int_{0}^{x}\frac{\Gamma\left(\alpha(x-u)\right)}{u^{\alpha(x-u)}}du,$
and by Leibniz integral rule, we have for all $x\in(0,T],$
\[
\frac{d}{dx}M(x)=\Gamma\left(\alpha(0)\right)x^{-\alpha(0)}+\int_{0}^{x}\frac{d}{dx}\left(\frac{\Gamma\left(\alpha(x-u\right)}{u^{\alpha(x-u)}}\right)du.
\]
 and 
\[
g(x)\frac{d}{dx}M(x)=g(x)M_{1}(x)+g(x)M_{2}(x).
\]
The derivative inside the integral in $M_{2}$ can be calculated explicitly
as follows, 
\[
\frac{d}{dx}\left(\frac{\Gamma\left(\alpha(x-u)\right)}{u^{\alpha(x-u)}}\right)=-u^{-\alpha(x-u)}\alpha'(x-u)\Gamma\left(\alpha(x-u)\right)\left(\ln\left(u\right)-\psi_{0}\left(\alpha(x-u)\right)\right),
\]
where $\psi_{0}$ is the digamma function. First we have that 
\[
|g(x)M_{1}(x)|\leq\Gamma\left(\alpha(0)\right)x^{\alpha(x)+\epsilon-\alpha(0)}\lesssim x^{\alpha_{*}+\epsilon-\alpha(0)},
\]
 which is in $L^{p}\left([0,T]\right)$ if $\left(\alpha_{*}+\epsilon-\alpha(0)\right)\times p>-1.$
\\
Secondly we look at 
\[
g(x)M_{2}(x)=g(x)\int_{0}^{x}-u^{-\alpha(x-u)}\alpha'(x-u)\Gamma\left(\alpha(x-u)\right)\ln\left(u\right)du
\]
\[
+g(x)\int_{0}^{x}-u^{-\alpha(x-u)}\alpha'(x-u)\Gamma\left(\alpha(x-u)\right)\psi_{0}\left(\alpha(x-u)\right)du=:g(x)M_{2,1}(x)+g(x)M_{2,2}(x)
\]
 The first term is well defined since $\alpha$ is a differentiable
function on a compact interval, and the gamma function is decreasing
on some compact sub interval of $(0,1)$, therefore, 
\[
|g(x)M_{2,1}(x)|\leq\parallel g\parallel_{\alpha(\cdot);[0,T]}\parallel\alpha\parallel_{C^{1}}\Gamma\left(\inf(\alpha(t))\right)\frac{x^{1+\alpha(x)-\alpha^{*}}\left(\left(\alpha^{*}-1\right)\ln(x)+1\right)}{\left(\alpha^{*}-1\right)^{2}},
\]
where $\alpha^{*}=\sup_{x\in[0,T]}\alpha(x)$ . The above is well
defined as an element of $L^{p}\left([0,T]\right)$ for any $p$.
At last, for $g(x)M_{2,2}(x)$ we get a very similar estimate as for
$M_{2,1}$ by using that the digamma function is bounded on a compact
sub interval of $(0,1),$ and in fact for $x\in[a,b]\subset(0,1)$
, $|\psi_{0}\left(x\right)|\leq|\psi_{0}\left(a\right)|$. Therefore
we have that 
\[
x\mapsto g(x)\frac{d}{dx}M(x)\in L^{p}\left([0,T]\right)
\]
 for any $p$ satisfying the inequality $\left(\alpha_{*}+\epsilon-\alpha(0)\right)\times p>-1$. 

Combining the results above, we obtain that the $G_{0}$ functional
in our representation is well behaved on the space $\mathcal{C}_{0}^{\alpha(\cdot)+\epsilon}\left(\left[0,T\right]\right)$
and satisfies
\[
\left|G_{0}(g)(x)\right|\leq C(T,\epsilon,\alpha;x)\parallel g\parallel_{\alpha(\cdot)+\epsilon;[0,T]}.
\]
 By the above lemma, we know that 
\[
x\mapsto G_{0}(x)=\frac{1}{B(\alpha(x),1-\alpha(x))}\frac{d}{dx}\left(\int_{0}^{x}\frac{\Gamma\left(\alpha(t)\right)g(t)}{\left(x-t\right)^{\alpha(t)}}dt\right)
\]
 is in $L^{p}\left(\left[0,T\right]\right)$ for a some well chosen
$p$.
\end{proof}
From the last result, we now have a concrete way to define a multifractional
derivative which is the inverse operator of the generalized Riemann-Liouville
integral. 
\begin{cor}
\label{Def. Multifractional Derivative}Under the assumptions from
Lemma \ref{lem:eksistens av G0 operator}, define the multifractional
derivative $D_{0+}^{\alpha}g$ of a function $g\in\mathcal{C}^{\alpha(\cdot)+\epsilon}\left([0,T]\right)$
to be the solution of the differential equation 
\[
f(x)=G_{0}(g)(x)+\int_{0}^{x}f(s)F(s,x)ds,
\]
where $G_{0}$ (dependent on $g$) and $F$ is given as in Lemma \ref{lem:eksistens av G0 operator}
and \ref{Let existence of beta derivative integral }. 
\end{cor}
\begin{rem}
The derivative operator that we have defined above is indeed the inverse
operator of the multifractional integral. To see this, assume for
a moment that for a function $g\in\mathcal{C}^{\alpha(\cdot)+\epsilon}$,
we have the derivative 
\[
D_{0+}^{\alpha}g(x)=f(x)\,\,\,for\,\,\,f\in L^{p}\left(\left[0,T\right]\right)
\]
then we can look at equation (\ref{eq:Multifractional ODE}), which
we now know is well posed, and we can go backwards in the derivation
method, which is found in the beginning of section \ref{sec:Multifractional-Calculus},
to find that 
\[
g(t)=\frac{1}{\Gamma\left(\alpha\left(t\right)\right)}\int_{0}^{t}(t-s)^{\alpha\left(t\right)-1}f(s)ds=\left(I_{0+}^{\alpha}f\right)(t)=\left(I_{0+}^{\alpha}\left(D_{0+}^{\alpha}g\right)\right)(t),
\]
and hence $g\in I_{0+}^{\alpha}L^{p}\left(\left[0,T\right]\right)$.

By simple iterations of the above ODE, we can get an explicit representation
of this multifractional derivative. 
\end{rem}
\begin{thm}
\label{thm.:representation and boundedness of derivative-1}The multifractional
derivative can be represented in $L^{p}\left([0,T]\right)$ as the
infinite sequence of integrals 
\[
D_{0+}^{\alpha}g(x)=G_{0}(g)\left(x\right)+\sum_{m=1}^{\infty}\int_{\triangle^{(m)}\left(0,x\right)}G_{0}(g)\left(s_{m+1}\right)F(s_{m+1},s_{m})\times...\times F\left(s_{1},x\right)ds_{m+1}...ds_{1},
\]
 for any $g\in\mathcal{C}_{0}^{\alpha(\cdot)+\epsilon}\left([0,T];\mathbb{R}\right)$
with $\alpha\in C^{1}\left([0,T],[a,b]\right)$ for $[a,b]\subset(0,1)$
and $\epsilon<1-\alpha^{*}$. Furthermore, assume that for some $p>1,$
the regularity function $\alpha$ satisfies the inequality 
\[
\left(\alpha_{*}+\epsilon-\alpha(0)\right)\times p>-1.
\]
Then we have the estimate 
\[
\left|D_{0+}^{\alpha}g(x)\right|\leq C(T,\epsilon,\alpha;x)\parallel g\parallel_{\alpha(\cdot)+\epsilon;[0,T]},
\]
where $x\mapsto C(T,\epsilon,\alpha;x)$ is an $L^{p}$ function for
any $p$ satisfying the inequality and hence $D_{0+}^{\alpha}g\in L^{p}\left([0,T]\right).$
\end{thm}
\begin{proof}
The representation is immediate as a consequence of Definition \ref{Def. Multifractional Derivative}.
We know from Lemma \ref{lem:eksistens av G0 operator} and \ref{Let existence of beta derivative integral }
that 
\[
\left|G_{0}(g)(x)\right|\leq C(T,\epsilon,\alpha;x)\parallel g\parallel_{\alpha(\cdot)+\epsilon;[0,T]}
\]
and 
\[
x\mapsto C(T,\epsilon,\alpha;x)\in L^{p}\left([0,T]\right)
\]
 for any $p$ satisfying the inequality $(\alpha_{*}+\epsilon-\alpha(0)\times p)>-1.$
\\
 We therefore obtain 

\[
\left|\int_{\triangle^{(m)}\left(0,x\right)}G_{0}(g)\left(s_{m+1}\right)F(s_{m+1},s_{m})\times...\times F\left(s_{1},x\right)ds_{m+1}...ds_{1}\right|
\]
\[
\leq\parallel F\parallel_{\infty;\triangle\left([0,T]\right)}^{m}\int_{\triangle^{(m)}\left(0,x\right)}\left|G_{0}(g)\left(s_{m+1}\right)\right|ds_{m+1}...ds_{1}
\]
\[
\leq C\parallel F\parallel_{\infty;\triangle\left([0,T]\right)}^{m}\parallel g\parallel_{\alpha(\cdot)+\epsilon}\int_{\triangle^{(m)}\left(0,x\right)}C(T,\epsilon,\alpha;s_{m+1})ds_{m+1}...ds_{1}.
\]
We can then use Cauchy's formula for repeated integration given by
\[
\int_{\triangle^{(m)}\left(0,x\right)}f(s_{m+1})ds_{m+1}...ds_{1}=\frac{1}{m!}\int_{0}^{x}(x-t)^{m}f(t)dt.
\]
 Using this, we get the estimate on the representation 
\[
|D_{0+}^{\alpha}g(x)|\leq C(T,\epsilon,\alpha;x)\parallel g\parallel_{\alpha(\cdot)+\epsilon;[0,T]}
\]
\[
+\parallel g\parallel_{\alpha(\cdot)+\epsilon}\sum_{m=1}^{\infty}\int_{0}^{x}\parallel F\parallel_{\infty;\triangle\left([0,T]\right)}^{m}\frac{(x-t)^{m}}{m!}C(T,\epsilon,\alpha;t)dt
\]
\[
\leq C(T,\epsilon,\alpha;x)\parallel g\parallel_{\alpha(\cdot)+\epsilon}\exp\left(\parallel F\parallel_{\infty;\triangle\left([0,T]\right)}x\right)
\]
 which completes the proof. 
\end{proof}
\begin{rem}
\label{rem: H\" {o}lder included in mutifractional -1}With the last Theorem
\ref{thm.:representation and boundedness of derivative-1}, we can
see that the multifractional derivative operator is well defined on
any local H\" {o}lder continuous space of order $\alpha(\cdot)+\epsilon$
for some small $\epsilon$, as long as $p$ and $\epsilon$ satisfies
$\left(\alpha_{*}+\epsilon-\alpha(0)\right)\times p>-1$. Therefore,
we have the relation 
\[
\mathcal{C}^{\alpha(\cdot)+\epsilon}\left(\left[0,T\right];\mathbb{R}\right)\subset I_{c+}^{\alpha}L^{p}\left([0,T]\right).
\]
\end{rem}

\subsection{Remark on multifractional Calculus}

Most of the articles we have found on the multifractional calculus
or fractional calculus of variable order has been related to applications
in physics. In an article by Hartley and Lorenzo \cite{Lorenzo2002}
the authors suggest many applications of multifractional calculus
in physics. Particular examples are given when physical phenomena
is modeled by a fractional differential equation, i.e. 
\[
D_{0+}^{\alpha}y(t)=f(t,y(t)),\,\,\,y(0)=y_{0}\in\mathbb{R},
\]
but the fractional order parameter $\alpha$ is dependent on a variable,
which again is dependent on time. An example could be that $\alpha$
was estimated on the basis of temperature, but temperature changes
in time. Therefore, they suggest that by using multifractional differential
operators, one can overcome this problem, as one would be able to
give a different differential order $\alpha$ to different times.
The multifractional calculus enables us to construct more accurate
differential equations for processes where the local time regularity
of the process is depending on time. Of course, this can also be generalized
to multifractional differential operators in space, where the spatial
variables of a system has time dependent local regularity.\\
\\
Although the multifractional calculus seems like a suitable tool for
differential equations described above, the soul concept of multifractional
or even fractional calculus can be very difficult to grasp, and use
in practice. When considering fractional calculus we usually consider
an operator behaving well (as inverse, etc. ) with respect to a fractional
integral. This makes the operator dependent on the initial point of
the integral, and it is not linear with respect to this initial point,
in the sense that $I_{a+}^{\alpha}I_{b+}^{\alpha}\neq I_{(a+b)+}^{\alpha}$.
Furthermore, we can not fractionally differentiate a constant (different
than $0$), as it is not contained in $I_{a+}^{\alpha}L^{p}$. There
is therefore a strong dependence of the derivative operator on the
choice of the integral (as we have seen it is actually completely
determined by the chosen integral). The intuition of the framework
is therefore far from that of the regular calculus of Newton and Leibniz,
and a more rigorous understanding of the properties of multifractional
calculus is needed to consider differential equations and partial
differential equations with such operators.  

\section{Riemann-Liouville multifractional Brownian Motion}

In this section we give some preliminary results relating to the multifractional
Brownian motion (mBm) with Riemann Liouville and present a Girsanov
transform for this process. The multifractional Brownian motion (fBm)
was first proposed in the 1990's by Peltier and L\`{e}vy-Veh\`{e}l in \cite{Peltier}
and independently by Benassi, Jaffard, Roux in \cite{BenJafRed}.
The process is non-stationary and on very small time steps it behaves
like a fractional Brownian motion. However, by letting the Hurst parameter
in the fractional Brownian motion be a function of time, the H\" {o}lder
regularity of the process is depending on time, and therefore it makes
more sense to talk about local regularities rather than global. The
process was initially proposed as a generalization with respect to
the fBm representation given by Mandelbrot and Van-Ness, that is,
the mBm was defined by 
\[
\tilde{B}_{t}^{h}=c(h_{t})\int_{-\infty}^{0}(t-s)^{h_{t}-\frac{1}{2}}-\left(-s\right)^{h_{t}-\frac{1}{2}}dB_{s}+c(h_{t})\int_{0}^{t}(t-s)^{h_{t}-\frac{1}{2}}dB_{s}=:\tilde{B}_{t}^{(1),h}+\tilde{B}_{t}^{(2),h},
\]
where $\left\{ B_{t}\right\} _{t\in[0,T]}$ is a real valued Brownian
motion, and $h:[0,T]\rightarrow(0,1)$ is a continuous function. Notice
in the above representation that $\tilde{B}_{t}^{(1),h}$ is always
measurable with respect to the filtration $\tilde{\mathcal{F}}_{0}$
(generated by the Brownian motion), as the stochastic process only
``contributes'' from $-\infty$ to $0.$ Therefore, we can think
of $\tilde{B}_{t}^{(2),h}$ as the only part which contributes to
the stochasticity of $\tilde{B}_{t}^{h}$ when $t>0$. The reason
why one also considers the process $\tilde{B}_{t}^{(1),h}$ when analyzing
regular fractional Brownian motions (in the case $h(t)=H$) is to
ensure stationarity of the process. However, when we are considering
the generalization $\tilde{B}_{t}^{h}$ above, when $h$ is not constant,
we do not get stationary of the process even though we consider the
a representation as the one above. We are therefore inclined to choose
$\tilde{B}_{t}^{(2),h}$ to be the multifractional noise we consider
in this article due to its very simplistic nature. This multifractional
process is often called in the literature the $Riemann$-$Liouville$
multifractional Brownian motion, inspired by the original definition
of the fractional Brownian motion defined by L\`{e}vy in the 1940's. The
Riemann-Liouville multifractional Brownian motion was first analyzed
by S. C. Lim in \cite{SCLim}, and is well suited to the use of multifractional
calculus, constructed above, in the analysis of differential equations
driven by this process. In this section we will recite some of the
basic properties of the Riemann-Liouville multifractional Brownian
motion from \cite{SCLim}, and then we will use the next section to
apply this process with respect to the Girsanov theorem. \\
\\
We begin by giving a proper definition of what we mean by a Riemann-Liouville
multifractional Brownian motion. 
\begin{defn}
\label{Def. multifractional Brownin motion}Let $\left\{ B_{t}\right\} _{t\in[0,T]}$
be a one dimensional Brownian motion on a filtered probability space
$\left(\Omega,\mathcal{F},P\right)$, and let $h:[0,T]\rightarrow[a,b]\subset(0,1)$
be a $C^{1}$ function. We define the Riemann Liouville multifractional
Brownian motion $(RLmBm)$ $\left\{ B_{t}^{h}\right\} _{t\in[0,T]}$
by 
\[
B_{t}^{h}=\frac{1}{\Gamma\left(h_{t}+\frac{1}{2}\right)}\int_{0}^{t}(t-s)^{h_{t}-\frac{1}{2}}dB_{s};\,\,t\geq0,
\]
 where $\Gamma$ is the Gamma function. The function $h$ is called
the regularity function of $B_{\cdot}^{h}$ . 
\end{defn}
With this definition, it is useful to know the co-variance function
of the process that we are interested in. Although this function is
slightly more complicated than the co-variance of a regular fBm, or
even more complicated than that proposed for mBm's in article \cite{AyCoVe},
the co-variance obtained here is explicit, and suits our purposes
well. The below proposition can also be found in \cite{SCLim}, and
the proof is a simple exercise in standard integration. 
\begin{prop}
\label{prop:(Covariance-of-RLmBm)}(Co-variance of RLmBm) let the
RLmBm be defined as above. Then for $0\leq s\leq t\leq T$, we have
\[
R^{h}\left(t,s\right)=E\left[B_{t}^{h}B_{s}^{h}\right]=\frac{t^{h_{t}-\frac{1}{2}}s^{h_{s}+\frac{1}{2}}}{\Gamma\left(h_{t}+\frac{1}{2}\right)\Gamma\left(h_{s}+\frac{3}{2}\right)}F\left(1;\frac{1}{2}-h_{t};h_{s}+\frac{3}{2};\frac{t}{s}\right)
\]
where $F$ is the hyper-geometric function. 
\end{prop}
\begin{proof}
We know that 
\[
E\left[B_{t}^{h}B_{s}^{h}\right]=E\left[\frac{1}{\Gamma\left(h_{t}+\frac{1}{2}\right)\Gamma\left(h_{s}+\frac{1}{2}\right)}\int_{0}^{t}(t-r)^{h_{t}-\frac{1}{2}}dB_{r}\int_{0}^{s}(s-r)^{h_{s}-\frac{1}{2}}dB_{r}\right]
\]
\[
=\frac{1}{\Gamma\left(h_{t}+\frac{1}{2}\right)\Gamma\left(h_{s}+\frac{1}{2}\right)}\int_{0}^{s}(t-r)^{h_{t}-\frac{1}{2}}(s-r)^{h_{s}-\frac{1}{2}}dr=\frac{t^{h_{t}-\frac{1}{2}}s^{h_{s}+\frac{1}{2}}}{\Gamma\left(h_{t}+\frac{1}{2}\right)\Gamma\left(h_{s}+\frac{3}{2}\right)}F\left(1;\frac{1}{2}-h_{t};h_{s}+\frac{3}{2};\frac{t}{s}\right)
\]
\end{proof}
When discussing fractional and multifractional Brownian motions, their
H\" {o}lder regularities seem to be of great interest. We will therefore
give some short comments on the regularity of the Riemann-Liouville
mBm as this is an uncommon representation. 
\begin{prop}
Let $\left\{ B_{t}^{h}\right\} _{t\in[0,T]}$ be a RLmBm process,
with $h\in C^{1}\left(\left[0,T\right]\right)$ with values in some
compact subset of $(0,1)$. Then we have global regularity 
\[
|B_{t}^{h}-B_{s}^{h}|\leq C|t-s|^{\min(\frac{1}{2},h_{*})}\,\,\,for\,\,\,all\,\,(t,s)\in\triangle^{(2)}\left(\left[0,T\right]\right)
\]
 and local regularity 
\[
\lim_{u\rightarrow0}\frac{|B_{t+u}^{h}-B_{t}^{h}|}{|u|^{h(t)}}<\infty\,\,\,P-a.s.\,\,\forall t\in[0,T].
\]
\end{prop}
\begin{proof}
Both the claims are thoroughly proved in the seminal paper \cite{Peltier}
by Peltier and Veh�l. This article uses the representation $\tilde{B}_{t}^{h}$
as mentioned above, but do split their arguments in to the case of
$\tilde{B}^{(1),h}$ and $\tilde{B}^{(2),h}$, such that the proof
is very simple to read. 
\end{proof}
For a longer discussion on the properties of the multifractional Brownian
motion, we refer to \cite{AyCoVe,BenJafRed,BouDozMar,Corlay,LebovitsLevy,LebVehHer,SCLim,SergCo}. 

\section{Girsanov theorem and Existence of Strong Solutions }

In this section we will apply the multifractional calculus that we
developed in Section \ref{sec:Multifractional-Calculus} to analyze
differential equations driven by multifractional noise. By multifractional noise (or multifractional Brownian motion), we will from here on out use the Riemann-Liouville multifractional Brownian motion, as discussed earlier. We do this
by the same techniques used in the article by Nualart and Ouknine
\cite{NuaOki}, but we must generalize those arguments to allow for
multifractional noise, using the multifractional derivative operator.
We start to prove a Girsanov theorem for multifractional Brownian
motion, and then we use this to construct weak solutions. We then
check for uniqueness in distribution, and find that this again implies
path-wise uniqueness. To construct strong solutions, the authors of
\cite{NuaOki} use approximation techniques to show that when $b$
is of linear growth, we can find an appropriate sequence of functions
which is converging to $b$ such that the solution of the corresponding
SDE is strong for any of the functions in the approximating sequence.
The corresponding solutions can then be shown to converge to a stochastic
process, which the authors show is a strong solution to the SDE. The
whole argument is based on the behavior of the drift function $b$
, and not on the multifractionality of the noise. Therefore, the existence
of strong solutions in our case follows by construction of weak solutions
and application of the results relating to strong solutions given
by Proposition 6, and 7 and Theorem 8 in \cite{NuaOki}. However,
as we want this article to be self contained, we include the propositions
and proofs which is given in \cite{NuaOki}, and we will make it very
clear what we have used from this article. 

\subsection{Girsanov's theorem}

As before, let us denote by $B_{t}^{h}$ the multifractional Brownian
motion with its natural filtration $\left\{ \mathcal{F}_{t};t\in[0,T]\right\} $.
We would like to show through a Girsanov theorem, that a perturbation
of $B^{h}$ with a specific type of function will give us a mBm under
another new probability measure. Similar to the derivation found in
\cite{NuaOki} in case of fBm, we consider the following perturbation
\[
\tilde{B}_{t}^{h}=B_{t}^{h}+\int_{0}^{t}u_{s}ds
\]
\[
=\frac{1}{\Gamma\left(h_{t}+\frac{1}{2}\right)}\int_{0}^{t}\left(t-s\right)^{h_{t}-\frac{1}{2}}dB_{s}+\int_{0}^{t}u_{s}ds
\]
\[
=\frac{1}{\Gamma\left(h_{t}+\frac{1}{2}\right)}\int_{0}^{t}\left(t-s\right)^{h_{t}-\frac{1}{2}}d\tilde{B}_{s},
\]
 where 
\[
\tilde{B}_{t}=B_{t}+\int_{0}^{t}D_{0+}^{h+\frac{1}{2}}\left(\int_{0}^{\cdot}u_{s}ds\right)\left(r\right)dr.
\]
 The multifractional derivative $D_{0+}^{h+\frac{1}{2}}\left(\int_{0}^{\cdot}u_{s}ds\right)$
in the last equation is well defined as an element in $L^{2}\left(\left[0,T\right]\right)$
as long as $\int_{0}^{\cdot}u_{s}ds\in I_{0+}^{h+\frac{1}{2}}L^{2}\left([0,T]\right).$
\begin{thm}
\label{thm:(Girsanov-theorem}(Girsanov theorem for mBm) Let $\left\{ B_{t}^{h}\right\} _{t\in[0,T]}$
be a Riemann-Liouville multifractional Brownian motion with regularity
function $h:[0,T]\rightarrow\left[a,b\right]\subset(0,\frac{1}{2})$
such that $h\in C^{1}$. Assume that

$(i)$ $\int_{0}^{\cdot}u_{s}ds\in\mathcal{C}^{h(\cdot)+\frac{1}{2}+\epsilon}\left([0,T]\right)\subseteq I_{0+}^{h+\frac{1}{2}}L^{2}\left([0,T]\right)$
a.s.

$(ii)$ $E\left[Z_{T}\right]=1$ for 
\[
Z(T):=\exp\left(\int_{0}^{T}D_{0+}^{h+\frac{1}{2}}\left(\int_{0}^{\cdot}u_{s}ds\right)\left(r\right)dB_{r}-\frac{1}{2}\int_{0}^{T}\left|D_{0+}^{h+\frac{1}{2}}\left(\int_{0}^{\cdot}u_{s}ds\right)\left(r\right)\right|^{2}dr\right).
\]

Then the stochastic process 
\[
\tilde{B}_{t}^{h}=B_{t}^{h}+\int_{0}^{t}u_{s}ds,
\]
 is an RLmBm under the measure $\tilde{P}$ defined by 
\[
\frac{d\tilde{P}}{dP}=Z(T).
\]
\end{thm}
\begin{proof}
The proof follows from the derivations above the theorem, and then
applying the regular Girsanov theorem to $\tilde{B}_{t}=B_{t}+\int_{0}^{t}D_{0+}^{h+\frac{1}{2}}\left(\int_{0}^{\cdot}u_{s}ds\right)\left(r\right)dr,$
showing that $\tilde{B}_{t}$ is a BM under $\tilde{P}.$
\end{proof}
\begin{rem}
The reason that we only look at $h:[0,T]\rightarrow\left[a,b\right]\subset(0,\frac{1}{2})$
is that the multifractional derivative is only constructed for $\alpha:\left[0,T\right]\rightarrow\left[c,d\right]\subset(0,1)$,
and as we need to look at $\alpha=h+\frac{1}{2}$ to consider the
RLmBm, we need to restrict $0<h(t)<\frac{1}{2}$ for all $t\in[0,T]$.
We are currently working on a way to consider the multifractional
derivative for any function $\alpha$, i.e. such that we can differentiate
to variable order with largest value above $1$ and smallest value
less than 1. 
\end{rem}

\subsection{Existence of weak solutions}

\begin{thm}
\label{thm:Existence of weak solutions} Let $b$ be integrable and
of linear growth, i.e. $|b(t,x)|\leq\tilde{C}(1+|x|)$ for all $(t,x)\in\left[0,T\right]\times\mathbb{R}$.
Then there exists a weak solution $\left\{ X_{t}^{x}\right\} _{t\in[0,T]}$
to equation 
\begin{equation}
X_{t}^{x}=x+\int_{0}^{t}b\left(s,X_{s}^{x}\right)ds+B_{t}^{h};\,\,X_{0}^{x}=x\in\mathbb{R},\label{eq:Weak Solution deff. eq.}
\end{equation}
where $B^{h}$ is given as in Definition \ref{Def. multifractional Brownin motion}
.
\end{thm}
\begin{proof}
Set $u_{s}=b\left(s,B_{s}^{h}+x\right),$ and $\tilde{B}_{t}^{h}=B_{t}^{h}-\int_{0}^{t}u_{s}ds$,
and 
\[
v_{t}=D_{0+}^{h+\frac{1}{2}}\left(\int_{0}^{\cdot}u_{s}ds\right)(t).
\]
 We need to check that the process $\int_{0}^{\cdot}u_{s}ds$ satisfies
condition $(i)$ and $(ii)$ in theorem \ref{thm:(Girsanov-theorem}.
First we will show $(i)$ , i.e. $\int_{0}^{\cdot}u_{s}ds\in I_{0+}^{h+\frac{1}{2}}L^{2}\left(\left[0,T\right]\right)$,
which is equivalent to showing $v\in L^{2}\left(\left[0,T\right]\right)$.\textbf{
}By theorem \ref{thm.:representation and boundedness of derivative-1},
we know that 
\[
\left|v_{t}\right|\leq\parallel\int_{0}^{\cdot}u_{s}ds\parallel_{h(\cdot)+\frac{1}{2}+\epsilon;[0,T]}C(T,\epsilon,h+\frac{1}{2};t),
\]
 where $t\mapsto C(T,\epsilon,h+\frac{1}{2};t)\in L^{2}\left(\left[0,T\right]\right)$
since\textbf{ }$\left(h_{*}+\epsilon-h(0)\right)\times2>-1$ for some
small $\epsilon>0.$ Moreover, by the linear growth of $b$ we have
the following estimate 
\begin{equation}
\parallel\int_{0}^{\cdot}u_{s}ds\parallel_{h(\cdot)+\frac{1}{2}+\epsilon;[0,T]}\leq\tilde{C}(1+|B^{h}|_{\infty}+|x|)T^{(\frac{1}{2}-h^{*}-\epsilon)},\label{eq:linear growth condition bound}
\end{equation}
and we know by Fernique that $|B^{h}|_{\infty}$ has finite exponential
moments of all orders, and therefore we have that 
\begin{equation}
t\mapsto v_{t}\in L^{2}\left([0,T]\right)\,\,\,P-a.s.\label{eq:bounded derivative in girsanov}
\end{equation}
Next we look at $(ii)$, where we need to have that $u$ is adapted
and satisfies the Novikov condition. It is sufficient to prove that
there exists a $\lambda>0$ such that 
\[
\sup_{0\leq s\leq T}E\left[\exp\left(\lambda v_{s}^{2}\right)\right]<\infty.
\]
The adaptedness of $v$ follows, as $D_{0+}^{h+\frac{1}{2}}$ is a
deterministic operator. By equation (\ref{eq:linear growth condition bound})
and Fernique's theorem, the above condition is satisfied, and the
existence of a weak solution follows from the Girsanov Theorem \ref{thm:(Girsanov-theorem}. 
\end{proof}
Let us continue by checking for uniqueness of the solution. 

\subsection{Uniqueness in law and path-wise uniqueness. }

We will prove uniqueness in law and path wise uniqueness in the same
manner as was done in \cite{NuaOki} section 3.3. The technique used
is very similar, although we have different bounds on our differential
equations, corresponding to the multifractality of our equation. 
\begin{thm}
\label{thmUniqueness of equation}Let $\left(X,B^{h}\right)$ be a
weak solution to differential equation (\ref{eq:Weak Solution deff. eq.})
defined on the corresponding filtered probability space $\left(\Omega,\mathcal{F},P\right)$.
Then the weak solution is unique in law, and moreover is unique almost
surely. 
\end{thm}
\begin{proof}
Start again with
\[
\tilde{u}_{t}=D_{0+}^{h+\frac{1}{2}}\left(\int_{0}^{\cdot}b(s,X_{s})ds\right)(t),
\]
and let a new probability measure $\tilde{P}$ be defined by the Radon-Nikodym
derivative 
\[
\frac{d\tilde{P}}{dP}=\exp\left(-\int_{0}^{T}\tilde{u}_{r}dB_{r}-\frac{1}{2}\int_{0}^{T}\tilde{u}_{r}^{2}dr\right),
\]
where $B$ denotes a regular Brownian motion. Although this Radon-Nikodym
derivative is similar as before, notice that we now have the solution
of the differential equation inside the function $b$. We will show
that this new process $\tilde{u}$ still satisfies condition $(i)$
and $(ii)$ of Girsanov Theorem \ref{thm:(Girsanov-theorem}. Since
we know that , 
\[
|X_{t}|\leq|x|+|\int_{0}^{t}b(s,X_{s})ds|+|B_{t}^{h}|
\]
\[
\leq|x|+\int_{0}^{t}\tilde{C}(1+|X_{s}|)ds+|B_{t}^{h}|,
\]
 we have by Gr�nwall's inequality that 
\[
|X_{\cdot}|_{\infty;[0,T]}\leq\tilde{C}\left(|x|+|B_{\cdot}^{h}|_{\infty}+T\right)\exp\left(\tilde{C}T\right)\,\,\,P-a.s.
\]
Moreover, by the estimates on the multifractional derivative from
Theorem \ref{thm.:representation and boundedness of derivative-1},
we know 
\[
|\tilde{u}_{t}|\leq\parallel\int_{0}^{\cdot}b(\cdot,s,X_{s})ds\parallel_{h(\cdot)+\frac{1}{2}+\epsilon;[0,T]}C\left(T,\epsilon,h+\frac{1}{2};t\right),
\]
and since 
\[
\parallel\int_{0}^{\cdot}b(\cdot,s,X_{s})ds\parallel_{h(\cdot)+\frac{1}{2}+\epsilon;[0,T]}\leq C(1+|X_{\cdot}|_{\infty;[0,T]})T^{(\frac{1}{2}-h^{*}-\epsilon)},
\]
we obtain the estimate 

\[
|\tilde{u_{t}}|\leq C(1+|X_{\cdot}|_{\infty;[0,T]})C\left(T,\epsilon,h+\frac{1}{2};t\right),
\]
 where $t\mapsto C\left(T,\epsilon,h+\frac{1}{2};t\right)\in L^{2}\left(\left[0,T\right]\right)$.
Combining the above estimates, and again using Ferniques theorem,
we have that Novikov's condition is satisfied and $E\left[\frac{d\tilde{P}}{dP}\right]=1$.
The classical Girsanov theorem then states that the process 
\[
\tilde{B}_{t}=\int_{0}^{t}u_{r}dr+B_{t},
\]
is an $\left\{ \mathcal{F}_{t}\right\} $- Brownian motion under $\tilde{P}$,
and we can then write 
\[
X_{t}=x+\int_{0}^{t}K_{h}(t,s)d\tilde{B}_{s},
\]
where $K_{h}(t,s)=\frac{1}{\Gamma(h_{t}+\frac{1}{2})}(t-s)^{h_{t}-\frac{1}{2}}$.
Now we have that $X_{t}-x$ is a $\left\{ \mathcal{F}_{t}\right\} $-multifractional
Brownian motion with respect to $\tilde{P}$ . Therefore, we must
have that $X_{t}-x$ and $\tilde{B}_{t}^{h}$ has the same distribution
under the probability $P$. We can show this by considering a measurable
functional $\varphi$ on $C\left([0,T]\right)$, we have 
\[
E_{P}\left[\varphi\left(X-x\right)\right]
\]
\[
=E_{\tilde{P}}\left[\varphi\left(X-x\right)\exp\left(\int_{0}^{T}D_{0+}^{h+\frac{1}{2}}\left(\int_{0}^{\cdot}b(s,X_{s})ds\right)(r)dB_{r}+\frac{1}{2}\int_{0}^{T}D_{0+}^{h+\frac{1}{2}}\left(\int_{0}^{\cdot}b(s,X_{s})ds\right)(r)^{2}dr\right)\right]
\]
\[
=E_{\tilde{P}}\left[\varphi\left(X-x\right)\exp\left(\int_{0}^{T}D_{0+}^{h+\frac{1}{2}}\left(\int_{0}^{\cdot}b(s,X_{s})ds\right)(r)d\tilde{B}_{r}-\frac{1}{2}\int_{0}^{T}D_{0+}^{h+\frac{1}{2}}\left(\int_{0}^{\cdot}b(s,X_{s})ds\right)(r)^{2}dr\right)\right]
\]
\[
=E_{P}\left[\varphi\left(B^{h}\right)\exp\left(\int_{0}^{T}D_{0+}^{h+\frac{1}{2}}\left(\int_{0}^{\cdot}b(s,x+B_{s}^{h})ds\right)(r)dB_{r}-\frac{1}{2}\int_{0}^{T}D_{0+}^{h+\frac{1}{2}}\left(\int_{0}^{\cdot}b(s,x+B_{s}^{h})ds\right)(r)^{2}dr\right)\right]
\]
\[
=E_{P}\left[\varphi\left(\tilde{B}^{h}\right)\right].
\]
 For the almost surely uniqueness, assume there exists two weak solutions
$X^{1}$ and $X^{2}$ on the same probability space $\left(\Omega,\mathcal{F},\tilde{P}\right)$,
then $\max\left(X^{1},X^{2}\right)$ and $\min\left(X^{1},X^{2}\right)$
are again two solutions, and according to the above proof, have the
same distribution. We can rewrite $\max(x,y)=\frac{1}{2}\left(x+y+|x-y|\right)$
and $\min(x,y)=\frac{1}{2}\left(x+y-|x-y|\right),$ and therefore
$E\left[\max\left(X^{1},X^{2}\right)\right]=E\left[\min\left(X^{1},X^{2}\right)\right]$
imply that 
\[
E_{\tilde{P}}\left[|X^{1}-X^{2}|\right]=-E_{\tilde{P}}\left[|X^{1}-X^{2}|\right],
\]
but of course this is only true if $X^{1}=X^{2}$ a.s. 
\end{proof}
\begin{rem}
Although, in this paper we limit our selves to the study of SDE's
with drift coefficient of linear growth, one can easily obtain weak
existence of time-singular Volterra equations of the form 
\[
X_{t}=x+\int_{0}^{t}V(t,s)b(s,X_{s})ds+B_{t}^{h},
\]
where $V$ is a singular and square integrable deterministic Volterra
kernel, and the drift $b$ is still of linear growth. As we have seen
in Theorem \ref{thmUniqueness of equation}, we only need to prove
that 
\[
\parallel\int_{0}^{\cdot}V(\cdot,s)b(s,X_{s})ds\parallel_{h+\frac{1}{2}+\epsilon}<\infty.
\]
But this can be verified in the case that $|V(t,s)|\leq C|t-s|^{h+\frac{1}{2}+\epsilon-1},$as
we then obtain $|\int_{0}^{t}V(t,s)ds|\leq C|t-s|^{h+\frac{1}{2}+\epsilon}.$
We are currently writing a paper studying such singular Volterra equations,
with merely linear growth conditions on the drift, and their relation
to stochastic fractional differential equations. 
\end{rem}

\subsection{\label{subsec:Existence-of-strong}Existence of strong solutions}

The following three theorems can be found in \cite{NuaOki} given
by Proposition 6, and 7 and Theorem 8, but is cited here with modifications
to accommodate the multifractional Brownian motion. 
\begin{prop}
\label{prop:Krylov inequality}Let $X$ denote a weak solution constructed
above, but where the drift $b$ is a uniformly bounded function. Fix
a constant $\rho>\sup_{t\in[0,T]}h_{t}+1$. For any measurable and
non-negative function $g$ with $g:[0,T]\times\mathbb{R}\rightarrow\mathbb{R},$
there exists a constant $K$ depending on $\parallel b\parallel_{\infty}$,
$\rho,$ and $T$ such that 
\[
E\left[\int_{0}^{T}g(t,X_{t})dt\right]\leq K\left(\int_{0}^{T}\int_{\mathbb{R}}g(t,x)^{\rho}dxdt\right)^{\frac{1}{\rho}}.
\]
\end{prop}
\begin{proof}
For compact notation, Let the co variance in proposition \ref{prop:(Covariance-of-RLmBm)}
be written as, 
\[
E[\left(B_{t}^{h}\right)^{2}]=c(h_{t})t^{2h_{t}}.
\]
Let $Z=\frac{d\tilde{P}}{dP}$ be the Radon-Nikodym derivative that
we constructed in Theorem \ref{thmUniqueness of equation}. Then by
the H\" {o}lder inequality for some $\frac{1}{\alpha}+\frac{1}{\beta}=1,$
we have the estimate
\[
E_{P}\left[\int_{0}^{T}g(t,X_{t})dt\right]=E_{\tilde{P}}\left[Z^{-1}\int_{0}^{T}g(t,X_{t})dt\right]
\]
\[
\leq\left(E_{\tilde{P}}\left[Z^{-\alpha}\right]\right)^{\frac{1}{\alpha}}\left(E_{\tilde{P}}\left[\int_{0}^{T}g(t,X_{t})^{\beta}dt\right]\right)^{\frac{1}{\beta}}.
\]
 The expectation of $Z^{-\alpha}$ is explicitly given for any $\alpha>1$
by 
\[
E_{\tilde{P}}\left[Z^{-\alpha}\right]=E_{\tilde{P}}\left[\exp\left(\alpha\int_{0}^{T}u_{s}dB_{s}+\frac{\alpha}{2}\int_{0}^{T}u_{s}^{2}ds\right)\right]
\]
\[
=E_{\tilde{P}}\left[\exp\left(\alpha\int_{0}^{T}u_{s}d\tilde{B}_{s}-\frac{\alpha}{2}\int_{0}^{T}u_{s}^{2}ds\right)\right]\leq K(\parallel b\parallel_{\infty},T,\alpha),
\]
obtained by arguments given in Theorem \ref{thmUniqueness of equation}.
Next we look at the other term, 

\[
E_{\tilde{P}}\left[\int_{0}^{T}g(t,X_{t})^{\beta}dt\right]
\]
\[
=\int_{0}^{T}\frac{1}{\sqrt{2\pi}c(h_{t})t^{h_{t}}}\int_{\mathbb{R}}g(t,x)^{\beta}\exp\left(-\frac{\left(y-x\right)^{2}}{2c(h_{t})t^{2h_{t}}}\right)dydt
\]
\[
\leq\frac{1}{\sqrt{2\pi}}\left(\int_{0}^{T}\int_{\mathbb{R}}g(t,y)^{\beta\gamma}dydt\right)^{\frac{1}{\gamma}}\left(\int_{0}^{T}t^{-h_{t}\gamma'}\int_{\mathbb{R}}e^{-\frac{\gamma'\left(x-y\right)^{2}}{2t^{2h_{t}}}}dydt\right)^{\frac{1}{\gamma'}},
\]
 where we applied the H\" {o}lder inequality again, with $\gamma=\frac{\gamma'}{\gamma'-1}\Leftrightarrow\frac{1}{\gamma}+\frac{1}{\gamma'}=1$
with $\sup_{t\in[0,T]}h_{t}+1<\gamma$. Explicit calculations then
yield 
\[
\leq\frac{1}{\sqrt{2\pi}}\left(\int_{0}^{T}\int_{\mathbb{R}}g(t,y)^{\beta\gamma}dydt\right)^{\frac{1}{\gamma}}\left(\int_{0}^{T}t^{\left(1-\gamma'\right)h_{t}}dt\right)^{\frac{1}{\gamma'}}\leq\frac{c(T,\gamma)}{\sqrt{2\pi}}\left(\int_{0}^{T}\int_{\mathbb{R}}g(t,y)^{\beta\gamma}dydt\right)^{\frac{1}{\gamma}},
\]
and the result follows. 
\end{proof}
The next proposition will show that if we are able to find a sequence
of bounded and measurable functions $\left\{ b_{n}\right\} _{n\in\mathbb{N}}$
converging to the drift function $b$ in equation (\ref{eq:Weak Solution deff. eq.})
almost surely, and the corresponding solutions $X^{(n)}$ to equation
(\ref{eq:Weak Solution deff. eq.}) constructed with $b_{n}$ converges
to some process $X$. Then $X$ is a solution to the original differential
equation (\ref{eq:Weak Solution deff. eq.}). When we have proved
this result, we will show that there actually exists such a sequence
of functions $\left\{ b_{n}\right\} _{n\in\mathbb{N}}$, when $b$
is of a certain class, and therefore, by the path wise uniqueness
property of the solution, we obtain that the solution to equation
(\ref{eq:Weak Solution deff. eq.}) is a strong solution. \\

\begin{prop}
\label{prop:convergence of drift b}Let $\left\{ b_{n}\right\} _{n\in\mathbb{N}}$
be a sequence of bounded and measurable functions on $[0,T]\times\mathbb{R}$
bounded uniformly by a constant $C$ such that 
\[
b_{n}(t,x)\rightarrow b(t,x)\,\,\,for\,\,\,a.a.\,\,(x,t)\in[0,T]\times\mathbb{R}.
\]
 Also, assume that the corresponding solutions $X_{t}^{(n)}$ of equation
(\ref{eq:Weak Solution deff. eq.}) given by 
\[
X_{t}^{(n)}=x+\int_{0}^{t}b_{n}\left(s,X_{s}\right)ds+B_{t}^{h},
\]
 converge to a process $X_{t}$ for all $t\in[0,T].$ Then $X_{\cdot}$
is a solution to equation \ref{eq:Weak Solution deff. eq.}. 
\end{prop}
\begin{proof}
We need to show that 
\[
E\left[\int_{0}^{T}\left|b_{n}\left(s,X_{s}^{(n)}\right)-b\left(s,X_{s}\right)\right|ds\right]\rightarrow0\,\,\,as\,\,\,n\rightarrow\infty.
\]
 Adding and subtracting $b_{n}\left(s,X_{s}\right)$ in the above
expression, we can find 
\[
E\left[\int_{0}^{T}\left|b_{n}\left(s,X_{s}^{(n)}\right)-b\left(s,X_{s}\right)\right|ds\right]
\]
\[
\leq E\left[\int_{0}^{T}\left|b_{n}\left(s,X_{s}\right)-b\left(s,X_{s}\right)\right|ds\right]
\]
\[
+\sup_{k\in\mathbb{N}}E\left[\int_{0}^{T}\left|b_{k}\left(s,X_{s}^{(n)}\right)-b_{k}\left(s,X_{s}\right)\right|ds\right]=:J_{1}(n)+J_{2}(n).
\]
Moreover, there exists a smooth function $\kappa:\mathbb{R}\rightarrow\mathbb{R}$
such that $\kappa(0)=1$ and $\kappa(z)\in[0,1]$ for all $|z|<1$
and $\kappa(z)=0$ for all $|z|\geq1$. Fix $\epsilon>0$ and let
$R$ be a constant such that 
\[
E\left[\int_{0}^{T}1-\kappa\left(\frac{X_{t}}{R}\right)dt\right]<\epsilon.
\]
By the Frechet-Kolmogorov theorem (a compactness criterion for $L^{2}\left(B\right)$
when $B$ is a bounded subset of $\mathbb{R}^{d}$), the sequence
of functions $\left\{ b_{k}\right\} _{k\in\mathbb{N}}$ is relatively
compact in $L^{2}\left(\left[0,T\right]\times\left[-R,R\right]\right)$,
and therefore we can find a finite set of smooth functions $\{H_{1},..,H_{N}\}\in L^{2}\left(\left[0,T\right]\times\left[-R,R\right]\right)$
such that for every $k$, 
\[
\int_{0}^{T}\int_{-R}^{R}|b_{k}(t,x)-H_{i}(t,x)|dxdt<\epsilon^{2}
\]
 for some $H_{i}\in\{H_{1},..,H_{N}\}$. Using this we can find 
\[
E\left[\int_{0}^{T}\left|b_{k}\left(s,X_{s}^{(n)}\right)-b_{k}\left(s,X_{s}\right)\right|ds\right]\leq E\left[\int_{0}^{T}|b_{k}\left(t,X_{t}^{(n)}\right)-H_{i}\left(t,X_{t}^{(n)}\right)|dt\right]
\]
\[
+E\left[\int_{0}^{T}|H_{i}(t,X^{(n)})-H_{i}\left(t,X_{t}\right)|dt\right]+E\left[\int_{0}^{T}|b_{k}\left(t,X_{t}\right)-H_{i}\left(t,X_{t}\right)|dt\right]
\]
\[
=:I_{1}(n,k)+I_{2}(n)+I_{3}(k).
\]
 We begin by looking at the components $I_{u}$, $u=1,2,3$ separately,
and start with $u=1.$ Using the $\kappa$ function we defined earlier,
we can see that 
\[
I_{1}(n,k)=E\left[\int_{0}^{T}\kappa\left(\frac{X_{t}}{R}\right)|b_{k}\left(t,X_{t}^{(n)}\right)-H_{i}\left(t,X_{t}^{(n)}\right)|dt\right]
\]
\[
+E\left[\int_{0}^{T}\left(1-\kappa\left(\frac{X_{t}}{R}\right)\right)|b_{k}\left(t,X_{t}^{(n)}\right)-H_{i}\left(t,X_{t}^{(n)}\right)|dt\right],
\]
then by Proposition \ref{prop:Krylov inequality}, we know that the
first of the two elements above is bounded in the following way, 
\[
E\left[\int_{0}^{T}\kappa\left(\frac{X_{t}^{(n)}}{R}\right)|b_{k}\left(t,X_{t}^{(n)}\right)-H_{i}\left(t,X_{t}^{(n)}\right)|dt\right]
\]
\[
\leq K\left(\int_{0}^{T}\int_{-R}^{R}|b_{k}\left(t,x\right)-H_{i}\left(t,x\right)|^{2}dxdt\right)^{\frac{1}{2}},
\]
 where $K_{1}$ depends on $\parallel b\parallel_{\infty}$ and $\sup_{i}\parallel H_{i}\parallel_{\infty}$
and $T$. We can then show for the second component, 
\[
E\left[\int_{0}^{T}\left(1-\kappa\left(\frac{X_{t}^{(n)}}{R}\right)\right)|b_{k}\left(t,X_{t}^{(n)}\right)-H_{i}\left(t,X_{t}^{(n)}\right)|dt\right]
\]
\[
\leq K_{2}E\left[\int_{0}^{T}\left(1-\kappa\left(\frac{X_{t}^{(n)}}{R}\right)\right)dt\right],
\]
where $K_{2}$ also depends on $\parallel b\parallel_{\infty}$ and
$\sup_{i}\parallel H_{i}\parallel_{\infty}$. Set $K=\max K_{1},K_{2}$,
by the properties of $H_{i}$ and $\kappa$ shown above in relation
to the sequence of $\left\{ b_{k}\right\} ,$ we have that, 
\[
\lim_{n\rightarrow\infty}\sup_{k}I_{1}(n,k)\leq K\epsilon.
\]
 In the same way, we can show that for a constant $K$ similar to
the one above, 
\[
\sup_{k}I_{3}(k)\leq K\epsilon,
\]
at last, of course $\lim_{n\rightarrow\infty}I_{2}(n)=0$ since $X^{(n)}\rightarrow X,$
and therefore, 
\[
\lim_{n\rightarrow\infty}J_{2}(n)=\lim_{n\rightarrow\infty}\sup_{k}\left(I_{1}(n,k)+I_{2}(n)+I_{3}(k)\right)=0.
\]
 For $J_{1}$ we can use a very similar argument, as we can decompose
it into 
\[
J_{1}(n)=E\left[\int_{0}^{T}\kappa\left(\frac{X_{t}}{R}\right)\left|b_{n}\left(s,X_{s}\right)-b\left(s,X_{s}\right)\right|ds\right]
\]
\[
+E\left[\int_{0}^{T}\left(1-\kappa\left(\frac{X_{t}}{R}\right)\right)\left|b_{n}\left(s,X_{s}\right)-b\left(s,X_{s}\right)\right|ds\right]
\]
 and use Proposition \ref{prop:Krylov inequality}, just as above. 
\end{proof}
The next theorem shows how we can combine Proposition \ref{prop:Krylov inequality}
and Proposition \ref{prop:convergence of drift b} to obtain strong
solutions when $b$ is of linear growth. Nualart and Ouknine do this
by constructing a proper sequence of functions which is bounded and
measurable such that this sequence converges to $b,$ and we can apply
Propositions \ref{prop:Krylov inequality}, and \ref{prop:convergence of drift b}. 
\begin{thm}
Let the drift $b$ in the SDE in equation (\ref{eq:Weak Solution deff. eq.})
be of linear growth, i.e. $|b(t,x)|\leq C(1+|x|)$ for a.a. $(t,x)\in[0,T]\times\mathbb{R}$.
Then there exists a unique strong solution to equation (\ref{eq:Weak Solution deff. eq.}). 
\end{thm}
\begin{proof}
The path-wise uniqueness is already obtained in Theorem \ref{thmUniqueness of equation}.
Therefore, the object of interest here is the strong existence. Define
a function $b_{R}(t,x)=b(t,\left(x\vee-R\right)\wedge R),$ which
by the linear growth condition we can see is bounded and measurable.
Next, let $\varphi$ be a non-negative test function with compact
support in $\mathbb{R}$ such that $\int_{\mathbb{R}}\varphi(y)dy=1$.
For $j\in\mathbb{N}$ define the function 
\[
b_{R,j}(t,x)=j\int_{\mathbb{R}}b_{R}(t,y)\varphi(j(x-y))dy,
\]
which one can verify is Lipschitz in the second variable uniformly
in $t$., Indeed, 
\[
|b_{R,j}(t,x)-b_{R,j}(t,y)|\leq j^{2}\sup_{t\in[0,T]}\int_{\mathbb{R}}b_{R}(t,u)du\times|x-y|.
\]
 Furthermore, define the functions 
\[
\tilde{b}_{R,n,k}=\bigwedge_{j=n}^{k}b_{R,j}\,\,\,and\,\,\,\tilde{b}_{R,n}=\bigwedge_{j=n}^{\infty}b_{R,j}.
\]
Both is again Lipschitz in the second variable uniformly in $t$,
and we see that $b_{R,n,k}\downarrow b_{R,n}$ as $k\rightarrow\infty$
for a.a. $x\in\mathbb{R}$ and $t\in[0,T]$.  Now, we can construct
a unique solution $\tilde{X}_{R,n,k}$ from equation (\ref{eq:Weak Solution deff. eq.})
with corresponding drift coefficient $b_{R,n,k}.$ By the comparison
theorem for ordinary differential equations, the sequence $\tilde{X}_{R,n,k}$
is decreasing as $k$ grows. Therefore $\tilde{X}_{R,n,k}$ has a
limit $\tilde{X}_{R,n}$. By the comparison theorem again, we have
that $\tilde{X}_{R,n,k}$ and $\tilde{X}_{R,n}$ are bounded from
above and below by $R$ and $-R$ respectively. Moreover, the solution
$\tilde{X}_{R,n}$ is increasing as $n$ gets larger, and converge
to a limit $X_{R}$. Now we can apply Proposition \ref{prop:convergence of drift b},
and we obtain that there exists a unique strong solution. 
\end{proof}

\section{Conclusion}

In this article we have constructed a new type of multifractional
derivative operator acting as the inverse of the generalized fractional
integral. We have then applied this derivative operator to construct
strong solutions to SDE's where the drift is of linear growth and
the noise is of non-stationary and multifractional type. The applications
of such equations may be found in a range of fields, including finance,
physics and geology. For future work, we are currently working on
a way to construct solutions to Volterra SDE's with singular Volterra
drift, and driven by self exciting multifractional noise. The methodology
will be similar to the above, but we will need to generalize the last
to sections to account for Volterra kernels of singular type. Furthermore,
we believe that the multifractional differential operator may shed
new light on both multifractional (partial) differential equations,
with possibly random order differentiation, and are currently working
on a project relating to such equations. 

\newpage{}

\section*{Appendix\label{sec:Apendix}}

In this section we will prove that the integral 
\[
\int_{0}^{x}K(x,t)g(t)dt:=\int_{0}^{x}\frac{\Gamma\left(\alpha(t)\right)g(t)}{\left(x-t\right)^{\alpha(t)}}dt,
\]
 satisfies for all $x>0,$ 

\[
\frac{d}{dx}\left(\int_{0}^{x}K(x,t)g(t)dt\right)
\]
\[
=g(x)\frac{d}{dx}\int_{0}^{x}K(x,t)dt-\int_{0}^{x}\frac{d}{dx}K(x,t)\left(g(x)-g(t)\right)dt.
\]
 The existence of the terms on the right hand side is shown in Lemma
\ref{lem:eksistens av G0 operator}, where it is shown that the right
hand side is elements of $L^{p}\left([0,T]\right)$, which is sufficient
for our application. To see the above relation, we shall use the definition
of the derivative. Start by adding and subtracting the point $g(x),$
and get
\[
\int_{0}^{x}K(x,t)g(t)dt=-\int_{0}^{x}K(x,t)\left(g(x)-g(t)\right)dt+g(x)\int_{0}^{x}K(x,t)dt=J_{1}(x)+J_{2}(x).
\]
 We then look at the right hand side as increments between $x+h$
to $x$, and after simple manipulations get
\[
J_{1}(x+h)-J_{1}(x)=\int_{0}^{x}K(x+h,t)\left(g(x+h)+\left(g(x)-g(x)\right)-g(t)\right)-K(x,t)\left(g(x)-g(t)\right)dt
\]
\[
+\int_{x}^{x+h}K(x+h,t)\left(g(x+h)-g(t)\right)dt
\]
\[
=\int_{0}^{x}\left(K(x+h,t)-K(x,t)\right)\left(g(x)-g(t)\right)dt+\left(g(x+h)-g(x)\right)\int_{0}^{x}K(x+h,t)dt
\]
\[
+\int_{x}^{x+h}K(x+h,t)\left(g(x+h)-g(t)\right)dt.
\]
 The second part is given by 
\[
J_{2}(x+h)-J_{2}(x)=g(x+h)\int_{0}^{x+h}K(x+h,t)dt-g(x)\int_{0}^{x}K(x,t)dt
\]
\[
=\left(g(x+h)-g(x)\right)\int_{0}^{x+h}K(x+h,t)dt+g(x)\int_{0}^{x+h}K(x+h,t)dt-g(x)\int_{0}^{x}K(x,t)
\]
\[
=\left(g(x+h)-g(x)\right)\int_{0}^{x}K(x+h,t)dt+\left(g(x+h)-g(x)\right)\int_{x}^{x+h}K(x+h,t)dt
\]
\[
+g(x)\left(\int_{0}^{x+h}K(x+h,t)dt-\int_{0}^{x}K(x,t)dt\right).
\]
 Combining $J_{1}$ and $J_{2}$ and we get and notice that the term
$\left(g(x+h)-g(x)\right)\int_{0}^{x}K(x+h,t)dt$ cancels out and
we obtain
\[
-\left(J_{1}(x+h)-J_{1}(x)\right)+J_{2}(x+h)-J_{2}(x)
\]
\[
=-\int_{0}^{x}\left(K(x+h,t)-K(x,t)\right)\left(g(x)-g(t)\right)dt+g(x)\left(\int_{0}^{x+h}K(x+h,t)dt-\int_{0}^{x}K(x,t)dt\right)
\]
\[
+\left(g(x+h)-g(x)\right)\int_{x}^{x+h}K(x+h,t)dt-\int_{x}^{x+h}K(x+h,t)\left(g(x+h)-g(t)\right)dt.
\]
 Let us first look at the last line above, 
\[
\left(g(x+h)-g(x)\right)\int_{x}^{x+h}K(x+h,t)dt-\int_{x}^{x+h}K(x+h,t)\left(g(x+h)-g(t)\right)dt=\int_{x}^{x+h}K(x+h,t)(g(t)-g(x))dt.
\]
we know, 
\[
|\frac{\Gamma\left(\alpha\left(t\right)\right)}{\left(x+h-t\right)^{\alpha\left(t\right)}}(g(x)-g(t))|\leq C\left(\frac{|x-t|^{\alpha\left(t\right)+\epsilon}}{|x+h-t|^{\alpha\left(t\right)}}\right)\rightarrow|x-t|^{\epsilon},
\]
 as $h\rightarrow0$. Further, we see
\[
\left|\int_{x}^{x+h}K(x+h,t)(g(t)-g(x))dt\right|\leq C\int_{x}^{x+h}\frac{|x-t|^{\alpha\left(t\right)+\epsilon}}{|x+h-t|^{\alpha\left(t\right)}}dt.
\]
Moreover, set $t=x+\lambda h$, and then we have 
\[
\int_{x}^{x+h}\frac{|x-t|^{\alpha\left(t\right)+\epsilon}}{|x+h-t|^{\alpha\left(t\right)}}dt=h^{1+\epsilon}\int_{0}^{1}|\lambda|^{\alpha(x+\lambda h)+\epsilon}|1-\lambda|^{-\alpha(x+\lambda h)}d\lambda.
\]
But the integral on the right hand side is bounded, i.e 
\[
\int_{0}^{1}|\lambda|^{\alpha(x+\lambda h)+\epsilon}|1-\lambda|^{-\alpha(x+\lambda h)}d\lambda\leq B\left(\alpha_{*}+\epsilon+1,1-\alpha^{*}\right).
\]
This imply that 
\[
\lim_{h\rightarrow0}\frac{\int_{x}^{x+h}K(x+h,t)(g(t)-g(x))dt}{h}=0.
\]
 We are left to look at the limits of the integrals 
\[
-\int_{0}^{x}\left(K(x+h,t)-K(x,t)\right)\left(g(x)-g(t)\right)dt\,\,\,and\,\,\,g(x)\left(\int_{0}^{x+h}K(x+h,t)dt-\int_{0}^{x}K(x,t)dt\right).
\]
 We will look at them separately, staring with the one one the left
above. By dominated convergence, we have 
\[
\lim_{h\rightarrow0}\frac{\int_{0}^{x}\left(K(x+h,t)-K(x,t)\right)\left(g(x)-g(t)\right)dt}{h}=\int_{0}^{x}\frac{d}{dx}K(x,t)\left(g(x)-g(t)\right)dt.
\]
Secondly, it is straight forward to see that, 
\[
\lim_{h\rightarrow0}\frac{g(x)\left(\int_{0}^{x+h}K(x+h,t)dt-\int_{0}^{x}K(x,t)dt\right)}{h}=g(x)\frac{d}{dx}\int_{0}^{x}K(x,t)dt.
\]
 The fact that $\frac{d}{dx}\int_{0}^{x}K(x,t)dt$ indeed exists is
proved in Lemma \ref{lem:eksistens av G0 operator} using the fact
that $K(x,t)=\frac{\Gamma\left(\alpha(t)\right)}{\left(x-t\right)^{\alpha(t)}}$. 

\bibliographystyle{plain}

\end{document}